\documentclass[12pt,a4paper]{article}

\usepackage{tocbibind}
\usepackage{preamble}

 \usepackage{vmargin}
 \setmargrb{1in}{0.7in}{1in}{0.7in} 

\title{Cactus barriers}
\author{\JaBu}

\date{11th February 2026}

\newenvironment{red}{\color{red}}{}
\newcommand{\bred}{\begin{red}}
\newcommand{\ered}{\end{red}}

\newenvironment{blue}{\color{blue}}{}
\newcommand{\bblue}{\begin{blue}}
\newcommand{\eblue}{\end{blue}}

\newenvironment{green}{\color{green}}{}
\newcommand{\bgreen}{\begin{green}}
\newcommand{\egreen}{\end{green}}

\usepackage[textsize=tiny]{todonotes}

\newcommand{\borderrank}[2][X]{\operatorname{br}_{#1}(#2)}
\newcommand{\cactusrank}[2][X]{\operatorname{cr}_{#1}(#2)}
\newcommand{\bordercactusrank}[2][X]{\operatorname{bcr}_{#1}(#2)}

\newcommand{\cactusop}[2]{\mathfrak{K}^{\circ}_{#1}\left( #2 \right)}

\begin{document}
\maketitle
\begin{abstract}
   Determinantal methods for bounding the rank and border rank of tensors or polynomials are subject to a major barrier. 
   For instance, it is known that using determinantal method one cannot prove a lower bound for the border rank of tensor in
   $\CC^m \otimes \CC^m \otimes \CC^m$  that exceeds $6m-4$.
   We explain the precise geometric reason for this number (and analogous bounds in more general tensor spaces) using cactus varieties and, more generally, scheme theoretic methods in algebraic geometry.
\end{abstract}

\medskip
{\footnotesize
\noindent\textbf{addresses:} \\
J.~Buczy\'nski, \eemail{jabuczyn@impan.pl}, 
   Institute of Mathematics of the Polish Academy of Sciences, ul. \'Sniadeckich 8, 00-656 Warsaw, Poland

\noindent\textbf{keywords:}\\
border rank, rank methods, cactus varieties, barriers, finite schemes

}

\section{Introduction}

Throughout the article $\kk$ is
an algebraically closed base field of any characteristic.

\subsection{Linear rank methods}
\label{sec_intro_linear_rank_methods}

The determinant of a square matrix, and more generally, the minors of a given order of any matrix $M$, are famously useful invariants throughout mathematics.
In the context of this article their relevance is given by the relation to a rank of $M$ and to bounds on tensor-type ranks.

In many areas of science, one frequently seeks to decompose a complicated state as a combination of simple ones.
More mathematically, given a vector space $W$ (the set of all potentially possible states, including the complicated ones) and its subset $\hat{X}\subset W$ of simple states we want to write an element $F\in W$ as a linear combination of vectors from $\hat X$ using as few such vectors as possible.
The minimal number of these elements is called $X$-rank, and it witnesses how complicated $F$ is.

There is a classical method of bounding the $X$-rank (and a related notion of $X$-border rank) of $F$ that involves a matrix $M$ of linear forms on $W$.
For each $w\in W$ we can evaluate each entry of $M$ at $w$ and obtain a matrix $M(w)$ whose coefficients are in $\kk$.
For each integer $k$ we can define the loci $\set{\rk M\leqslant k}=\set{w\in W \mid \rk M(w) \leqslant k}$.
Now suppose $k:= \max\set{ \rk(M(x)) \mid x\in \hat{X}}$, so that $\hat X\subset \set{\rk M\leqslant k}$.
If the $X$-rank of $F$ is $r$, then $\rk(M(F))\leqslant k \cdot r$, or vice versa, if $\rk(M(F))> k \cdot r$ for some integer $r$, then the $X$-rank of $F$ must be at least $r+1$.
Thus each matrix $M$ of linear forms gives us a lower bound for $X$-rank (and in fact also for the $X$-border rank).
These types of lower bounds are called in the literature ``rank methods for complexity'', see for instance \cite{efremenko_garg_oliveira_wigderson_barriers_for_rank_methods}.
In view of \cite{dolezalek_michalek_nonlinear_methods_for_tensors_det_eq_for_secants}
 in this paper we will call them \emph{linear rank methods} instead.

Since varying the matrix $M$ one gets different lower bounds,
a considerable effort has been devoted to finding the optimal matrix and bound for a given problem, see \cite{landsberg_ottaviani_VB_method_equ_for_secants} for an overview.
However, at some point the rank methods stop being effective,
and eventually barriers for the tensor cases were proven by \cite{efremenko_garg_oliveira_wigderson_barriers_for_rank_methods},
and \cite{garg_makam_oliveira_wigderson_more_barriers_for_rank_methods}.

\subsection{Border and cactus ranks}
\label{sec_intro_borders_and_cacti}

In this  article we present the geometric explanation for the barriers,
and generalise it to any other setting outside  the tensor world.
This geometric explanation has the name \emph{cactus variety} and our arguments build on earlier work, including \cite{nisiabu_jabu_cactus},
\cite{bernardi_ranestad_cactus_rank_of_cubics},
\cite{galazka_mgr_publ}
\cite{galazka_vb_cactus},
\cite{galazka_phd}.

In order to achieve our goal we move to the world of algebraic geometry and formalise the setting.
In Appendix~\ref{sec_appendix} we explain why doing this we do not lose any generality compared to the situation outlined above.

From the perspective of algebraic geometry it is more convenient to work in the projective space $\PP(W) = (W\setminus\set{0})/\kk^*$ instead of the vector space $W$.
On $\PP(W)$ we consider the Zariski topology.
For a subset (or subscheme) $R\subset \PP(W)$ by $\linspan{R}\subset \PP (W)$ we denote the projective linear span of $R$, that is the smallest projective linear subspace of $\PP(W)$ that contains $R$.
Equivalently, if $I(R)$ is the homogeneous ideal of all the functions vanishing on $R$,
then $\linspan{R}= Z(I(R)_1)$.

Suppose $X\subset \PP(W)$ is a locally closed subvariety.
We assume that $\linspan{X}=\PP(W)$.
For each positive integer $r$ define the following subsets of $\PP(W)$:
\begin{subequations}
\begin{align}
   \sigma^{\circ}_r(X)&:= \bigcup\set{\linspan{\fromto{x_1}{x_r}}  \mid x_i\in X}, \\
   \sigma_r(X)&:=\overline{\sigma^{\circ}_r(X)},\\
   \cactusop{r}{X}&:= \bigcup\set{\linspan{R}  \mid R \text{ is a finite subscheme of $X$ of degree } r},
   \label{equ_define_plain_cactus}
   \\
   \cactus{r}{X}&:=\overline{\cactusop{r}{X}}.
   \label{equ_define_cactus}
\end{align}
\end{subequations}
These sets allow us to define, respectively, the $X$-rank, $X$-border rank, $X$-cactus rank, $X$-border cactus rank.
\begin{defin}\label{def_ranks}
   With $X$ as above,
     for any $F\in \PP(W)$ define:
   \begin{itemize}
    \item the $X$\emph{-rank} of $F$:
    $r_X(F):= \min\set {r\in \NN \mid F\in \sigma^{\circ}_r(X)}$,
    \item the $X$\emph{-border rank} of $F$:
    $\borderrank{F}:= \min\set {r\in \NN \mid F\in \sigma_r(X)}$,
    \item the $X$\emph{-cactus rank} of $F$:
    $\cactusrank{F}:= \min\set{r\in \NN \mid F\in \cactusop{r}{X}}$,
    \item the $X$\emph{-border cactus rank} of $F$:
    $\bordercactusrank{F}:= \min\set{r\in \NN \mid F\in \cactus{r}{X}}$.
   \end{itemize}
   If $F\in W$ instead, and $[F]\in \PP(W)$ is its projective class,
   then we define $X$-rank of $F$ to be $r_X([F])$, and analogously for all the border and cactus variants.
   By convention, $r_X(F)$ (and $\borderrank{F}$, $\cactusrank{F}$, $\bordercactusrank{F}$)
   is $0$ if and only if $F=0\in W$, and it is never $0$ for $F\in \PP(W)$.
\end{defin}

We always have
$r_X(F) \geqslant \borderrank{F} \geqslant \bordercactusrank{F}$
and
$r_X(F) \geqslant \cactusrank{F} \geqslant \bordercactusrank{F}$.
Moreover the cactus variants of ranks are typically much smaller than the corresponding $X$-rank or $X$-border rank.
The main result of this article is that in cases of interest the linear rank methods outlined in Subsection~\ref{sec_intro_linear_rank_methods} can at best detect
$\bordercactusrank{F}$, and never $\borderrank{F}$ (unless when they coincide).

\subsection{Barriers in terms of cacti}

Let $X_0\subset X$ be the smooth locus of $X$ and let $ \overline{X}\subset \PP(W)$ be the closure of $X$.
In most cases of interest (such as Segre-Veronese, or Grassmannian varieties) $X$ is itself smooth and projective and thus $X_0=X=\overline{X}$.
Suppose further $M$ is a matrix of linear forms on~$W$.
By slight abuse of notation, we write $\set{\rk M \leqslant k} \subset \PP(W)$, the rank locus, also in the projective space.

In the notation of Subsection~\ref{sec_intro_borders_and_cacti},
the linear rank method outlined in Subsection~\ref{sec_intro_linear_rank_methods} can be rephrased as the following statements:
if $X\subset \set{\rk M \leqslant k}$ then all of the following statements hold:
\begin{align*}
   \forall_{F\in W} \ r_X(F) & \geqslant \frac{\rk M(F)}{k}, &
   \forall_{r\in \NN} \ \sigma^{\circ}_r(X)  & \subset \set{\rk M \leqslant k\cdot r},
   \\
\forall_{F\in W} \ \borderrank{F} & \geqslant \frac{\rk M(F)}{k},
&
   \forall_{r\in \NN} \ \sigma_r(X)  & \subset \set{\rk M \leqslant k\cdot r}.
\end{align*}

\begin{thm}[Cactus barrier for linear rank method]
   \label{thm_cactus_barriers_intro}
   With variety $X$ and matrix $M$ as above, if in addition $X\subset \set{\rk M \leqslant k}$
   then $\cactus{r}{X_0} \subset \set{\rk M \leqslant k \cdot r}$.
   In particular, if $g$ is an integer such that
   $\cactus{g}{X_0} = \PP(W)$ (for instance, $g$ is the generic $X_0$-cactus rank), then the linear rank methods will never provide a better lower bound on $X$-border rank than $g$.
\end{thm}

The exact value of $X$-generic rank (thus the maximal possible $X$-border rank) is known when $X$ is a Veronese variety (the famous Alexander-Hirschowitz Theorem \cite{alexander_hirschowitz_polynomial_interpolation}, \cite{brambilla_ottaviani_on_AH_theorem}), some families of Segre varieties \cite{abo_ottaviani_peterson_segre}, two factor Segre-Veronese varieties \cite{abo_brambilla_galuppi_oneto_nondef_of_Segre_Veronese}, \cite{dolezalek_ken_secant_vars_Segre_Veronese_O_1_2_nondef} and in further sporadic cases.
Moreover, it is straightforward to notice
that it is at least $\frac{\dim W}{\dim X+1}$.
On the other hand, generic cactus rank is rarely known, and usually much lower than $\frac{\dim W}{\dim X+1}$.
For instance, for the Segre embedding
$\PP^{a-1}\times\PP^{b-1}\times \PP^{c-1}\subset\PP^{abc-1}$,
it is known that for $g=2(a+b+c-2)$ the $g$-th cactus variety $\cactus{g}{X} = \PP^{abc-1}$, while we need at least $\frac{abc}{a+b+c-2}$-th secant variety to fill in $\PP^{abc-1}$.
Thus for large values of $a$, $b$, $c$, the gap between the border rank of most tensors and the best possible lower bound that we can obtain using linear rank methods is huge.
Analogous gaps show up in all the other tensor spaces provided the dimension of $X$ is sufficiently large and the embedding is sufficiently involved.

Combining Theorem~\ref{thm_cactus_barriers_intro}
with known results about generic cactus rank \cite[Cor.~1.7]{galazka_mgr_publ} (and earlier work)
recovers the asymptotic barriers for linear rank methods \cite[Thms~1 and 2]{efremenko_garg_oliveira_wigderson_barriers_for_rank_methods}, \cite[Thm.~8.4, Cor.~8.5]{garg_makam_oliveira_wigderson_more_barriers_for_rank_methods}.

\subsection{Prior work}
Weaker versions
of Theorem~\ref{thm_cactus_barriers_intro}
appeared previously since the introduction of cactus varieties.
In particular, in \cite[Prop.~3.6]{nisiabu_jabu_cactus}
the claim is proven for $M$ arising
as a catalecticant matrix (over complex numbers).
Later \cite[Thm~5]{galazka_vb_cactus} proved the barrier whenever $M$ arises
from so called vector bundle construction \cite[Constr.~2]{galazka_vb_cactus}.
Finally,
\cite[Thm~1.18]{galazka_phd} presents the statement whenever $X$ is contained in the constant rank locus,
attributing it to the author of this article.

\subsection{Strengthening of barriers}

As consequences of results of this paper we outline further possibilities of research direction.
Firstly, so far all the lower bounds on the border rank that were discovered seem to also be bounds on border cactus rank, with two notable exceptions discussed in Subsection~\ref{sec_barrier_breaking}.
That is, cactus barriers probably apply to much wider class of methods, than just linear rank methods.
For instance, the flag conditions
\cite[Cor.~2.3]{landsberg_michalek_abelian_tensors},
\cite[Prop.~2.3 and 2.5]{conner_huang_landsberg_bad_and_good_news}
are not in any obvious way related to the rank method,
but, again, the lower bounds for border rank obtained in this way are also lower bounds for border cactus rank.
Analogously,
the lower bound on Waring rank \cite[Prop.~4]{jabu_teitler_examples_of_forms_of_high_rank}
is also subject to the cactus barriers. Indeed,
the lower bound is never higher than the length of an apolar scheme outside of a specific hyperplane by \cite[Rem.~6]{jabu_teitler_examples_of_forms_of_high_rank}.
Furthermore, \cite[Thm~3]{bernardi_ranestad_cactus_rank_of_cubics} construct such a scheme of relatively small length and this is precisely the upper bound for generic cactus rank.
Also Weak Border Apolarity
\cite[Thm~1.2]{nisiabu_jabu_border_apolarity}, which proved to be more effective than many earlier rank based lower bounds,
is subject to the barriers by the results of
\cite[Thm~1.2]{nisiabu_jabu_abcd}.

Thus an important open problem is to determine whether all other existing methods for bounding rank or border rank are subject to some versions of the cactus barrier.

Another direction would be strengthening the cactus bounds. Focusing on tensors, or more generally on $X=\PP(A) \times Y \subset \PP(A\otimes V)$ with $Y\subset \PP(V)$
there is a well known correspondence between
secant varieties of $X$ and Grassmann secant varieties $\sigma_{r,i}(Y)$ of $Y$.
See \cite{chiantini_coppens_Grassmannians_of_secant_varieties} or \cite{landsberg_jabu_ranks_of_tensors} or Section~\ref{sec_cacti} for definition of Grassmann secant varieties.
In particular,
if  $\dim A\leqslant \dim V$,
$\sigma_r(X)=\PP(A\otimes V)$
if and only if $r\geqslant\dim A$ and $\sigma_{r,\dim A}(Y)=\Gr(\dim A, V)$.
However, this correspondence does not generalise to Grassmann cactus varieties
(Definition~\ref{def_Grassmann_cactus_variety}).
If $\cactus{r}{X}=\PP(A\otimes V)$,
  then  $\cactus{r,\dim A}{Y}=\Gr(\dim A, V)$,
  but not necessarily vice versa.
For instance, if $\dim A = m$
and $Y=\PP^{m-1}\times \PP^{m-1} \subset \PP(\kk^{m}\otimes \kk^{m})$,
so that
\[
   X = \PP^{m-1}\times \PP^{m-1}\times \PP^{m-1} \subset \PP(\kk^{m}\otimes \kk^{m}\otimes \kk^m)
\]
 then we know that
$\cactus{g_1}{X}=\PP(A\otimes V)$ for $g_1=6m-4$, while  $\cactus{g_2,m}{Y}=\Gr(m, V)$ already for $g_2= 3m-1$.
It seems that $3m-1$ is closer to the best lower bounds for tensor rank known for any family of large tensors \cite{landsberg_michalek_hay_in_haystack}.
This suggests the question if there is any version of the barriers that involves Grassmann cactus variety instead of cactus variety.
With this purpose in mind we prove some of our intermediate results in the more general setting of Grassmann cactus varieties.

\subsection{Barrier breaking}
\label{sec_barrier_breaking}

Breaking the cactus barrier---that is, finding a method or bound for border rank that exceeds the analogous bound on border cactus rank---is strictly related to the problem of smoothability of finite schemes,
see \cite{jelisiejew_PhD} for an overview, or more generally with problems in deformation theory.
In particular, any result on smoothability should imply a corresponding result on breaking the barrier.
Conversely, any barrier breaking has consequences in deformation theory.

To the author's knowledge there are currently only two results breaking the cactus barrier, both groundbreaking.
The first one \cite{galazka_mandziuk_rupniewski_distinguishing}
exploits the known facts on smoothability and translates them to the language of tensors.
The second one \cite{dolezalek_michalek_nonlinear_methods_for_tensors_det_eq_for_secants}
does not rely on deformation theory, instead hopefully in the future it will bring new information about smoothability of finite schemes.

\subsection*{Acknowledgements}
I am sincerely grateful to
Alessandra Bernardi,
Weronika Buczy{\'n}ska,
Klim Efremenko,
Maciej Ga{\l}{\k a}zka,
Joachim Jelisiejew,
Joseph Landsberg,
Visu Makam,
Mateusz Micha{\l}ek,
Kristian Ranestad, and
Zach Teitler
  for many inspiring fruitful discussions and helpful comments.
The author is supported by the National Science Center, Poland, project ``Advanced problems in contact manifolds, secant varieties, and their generalisations (APRICOTS+)'', 2023/51/B/ST1/02799.

\section{Finite subschemes}\label{sec_finite_schemes}

In this section we gather facts about schemes that are needed to prove our main result.
We denote these schemes by $R$, perhaps with some decoration (for instance, $R'$, or $R_1$, $R_2$,\dots, or $R_t$ when we consider a family of schemes parametrised by $t\in T$ for some base $T$).
Our main interest is in the case $R$ is a finite scheme over $\kk$ and has a fixed embedding into a projective space $\PP(W)$. Some statements cover a more general setting.

Suppose $X$ and $T$ are schemes over $\kk$.
By a \emph{flat family of subschemes} of $X$ parameterised by $T$ we mean a closed subscheme $\ccR \subset X\times T$ such that the projection $\pi\colon \ccR \to T$ is flat and surjective.
In this setting for each $\kk$-point $t\in T$
we denote by $R_t\subset X$ the fibre $\pi^{-1}(t)$.
We frequently distinguish a special $\kk$-point $0\in T$, and consider $R_0$ to be the special fibre.
By \emph{curve} we mean a smooth quasiprojective integral scheme over $\kk$ of dimension $1$.

The first claim is that on a smooth variety we can deform any finite subscheme into any open dense subset.
It appears to be a standard folklore result for which we have not found a reference.

\begin{prop}
   \label{prop_moving_scheme_into_generic_position}
   Suppose $X$ is a smooth quasiprojecive scheme over $\kk$, $U\subset X$ is an open dense subset,
   and $R\subset X$ is a finite subscheme.
   Then there exists a curve $T$, a $\kk$-point $0\in T$ and a finite flat family $\ccR \to T$ of subschemes of $X$ such that $R_0=R$ and for all $t\in T\setminus {0}$ we have $R_t\subset U$.
\end{prop}

\begin{prf}
   We first prove the claim when $R$ is local, that is,  its support is a single
   $\kk$-point $x\in X$.
   Without loss of generality,
     replacing $X$ by the connected component containing~$x$,
     we may suppose
     $X$ is connected, and let $n=\dim X$.
   We may also suppose $X$ is affine by replacing it by any of its open affine subsets containing $x$ (and $U$ is replaced by the intersection with this open affine subset).
   Since $X$ is smooth at $x$ there exists a dominant map
   $\phi\colon X\to \AAA^n$ (defined by the local coordinate functions near $x$)
   such that $\phi(x)=0\in \AAA^n$ and $\phi$ is etale at $x$.
   In particular, $\phi|_{R}$ is an isomorphism onto its image $\phi(R)\subset \AAA^n$.
   Again replacing $X$ by its open subset, we may assume in addition that $\phi$ is etale on $X$.
   Note that $\phi(U)$ is open by
   \cite[\href{https://stacks.math.columbia.edu/tag/03WT}{Lemma 03WT}]{stacks_project}.
   Pick any $\kk$-point $v\in \phi(U)$ considered as a vector in $\kk^n$,
   let $\tau \colon \GG_a \times \AAA^n \to \AAA^n$ be the translation action by multiples of $v$:
   \[
     \tau (\lambda, p):= p + \lambda v.
   \]
  Let $\ccR' \simeq \GG_a\times \phi(R)$, that is $\ccR'$ is a flat family of subschemes of $\AAA^n$ with each $R'_{\lambda}$ isomorphic to $R$, supported at $\lambda v$ and being the translation of $\phi(R)$ by $\lambda v$.

  Let $T' = \phi^{-1} (\tau(\GG_a \times \set{0}))$
    be the preimage of the curve (line)
    $\tau(\GG_a \times \set{0})\subset \AAA^n$.
  Note that since $\phi$ is etale,
     $T'$ is a smooth curve, and it intersects $U$ at a preimage of $v$.
  Thus the intersection $T'\cap (X\setminus U)$ is a finite collection of points.
  Exclude $x$ from this collection, and set $T\subset T'$ to be the open subset avoiding this collection.
  Now let $\ccR$ be the pullback of $\ccR'$ along the map $T\to \GG_a$ and let $0:=x \in T\subset X$.
  It is straightforward to verify that $\ccR$ has the desired properties, and thus concluding the proof when $R$ is local.

  Now suppose $R = R_1\sqcup \dotsb \sqcup R_k$
  with each $R_i\subset X$ a finite local subscheme supported at $x_i$.
  We find the curves $T_i$, each with a point $0_i\in T_i$ and a family $\ccR_i \to T_i$, as in the conclusion of the proposition for each local component $R_i$.
  Consider the product $S:=T_1\times \dotsb \times T_k$ with
  $0 = (\fromto{0_1}{0_k}) \in S$.
  Let
  \[
    D := \bigcup_{i =1}^{k} \Bigl(T_1\times \dotsb \times  T_{i-1}\times \set{0_i}\times T_{i+1}\times \dotsb \times T_k\Bigr)\subset S.
  \]
  Since $T_i$ and $S$ are quasiprojecive,
    we can find a curve $T'\subset S$ such that
    $0\in T'$ and $T'$ is not contained in~$D$.
  Similarly to the local case, by getting rid of redundant points of intersection,
  pick $T\subset T'$  to be an open subset such that $T\cap D = \set{0}$ (perhaps with some multiplicity).
  We have the restricted projection morphisms $T\to T_i$, and we can define
  Let $\ccR\subset T\times X$ to be the union of the pullbacks of the families $\ccR_i$.
  Again, this $T$ and $\ccR$ satisfies the conclusions of the proposition.
\end{prf}

\begin{lemma}\label{lem_bundles_are_trivial}
   Suppose $R$ is a scheme, which is either:
   \begin{itemize}
    \item a local scheme, or
    \item a finite scheme.
   \end{itemize}
   Moreover, suppose $\ccE$ is a vector bundle of rank $k$ on $R$.
   Then $\ccE$ is trivial, that is $\ccE \simeq \ccO_R^{\oplus k}$.
\end{lemma}

\begin{proof}
   First suppose $R$ is local, in particular, it has only one closed point $p\in R$ and any (non-empty) closed subscheme of $R$ contains $p$.
   Let $U \subset R$ be any open neighbourhood of $p$.
   Then the complement of $U$ is a closed subscheme which does not contain $p$ thus the complement is empty and $U = R$.
   The locally trivial sheaf $\ccE$ is trivial on some open neighbourhood of $p$, hence it is trivial on $R$.

   Now if $R$ is finite (or more generally a discrete union of local schemes),
     then the restriction of $\ccE$ to each connected component $R'\subset R$
     is trivial, and isomorphic to $\ccO_{R'}^{\oplus k}$.
   The trivialisations can be glued together to $\ccE\simeq \ccO_{R}^{\oplus k}$.
\end{proof}

\begin{lemma}\label{lem_exists_B_prime}
   Suppose $R$ is a finite scheme of length $r$, $B$ is a vector space, and $\ccE$ is vector bundle of rank $k$ on $R$.
   Assume $\varphi \colon \ccE \hookrightarrow \ccO_R\otimes B$ is an embedding of vector bundles on $R$.
   Then there exists a linear subspace $B' \subset B$ of dimension at most $k \cdot r$ such that $\varphi$ factorises
     $\ccE \hookrightarrow \ccO_R \otimes B'\hookrightarrow \ccO_R\otimes B$.
\end{lemma}

\begin{proof}
   The vector bundle $\ccE$ is trivial by Lemma~\ref{lem_bundles_are_trivial}, since $R$ is a finite scheme.
   First we prove the statement for $k=1$, that is $\ccE \simeq \ccO_R$.
   Passing to the projectivisations (which is possible, since $\varphi$ is an embedding) we get a section $R \to R \times \PP B$.
   Consider the image under the projection to $\PP B$.
   The image of the composed map is a finite subscheme in $\PP B$ of length at most $r$.
   Therefore its linear span $\PP (B') \subset \PP B$ is a (projective) linear subspace of dimension at most $r-1$, that is $\dim B'\le r$,
      and $\varphi$ factorises as claimed in the lemma.

   If $k>1$ is arbitrary, that is $\ccE\simeq \ccO_R^{\oplus k}$, then for each copy of $\ccO_R$ we have a vector space of dimension at most $r$.
   Taking as $B'$ the algebraic sum of these vector space we obtain the claim.
\end{proof}

Let $A$ and $B$ be vector spaces.
On the projective space $\PP(A\otimes B)$ we have the natural map
     $\phi\colon\ccO_{\PP (A\otimes B)} \otimes A^* \to \ccO_{\PP (A\otimes B)}(1) \otimes B$
     of vector bundles.
Explicitly, in coordinates, suppose $\set{\alpha_i}$, $\set{b_j}$, $\set{\beta_j}$ and $\set{\gamma_{ij}}$ are,  respectively, bases of $A^*$, $B$, $B^*$, and $A^*\otimes B^*$.
Suppose further, that $\set{\beta_j}$
  is the dual basis to $\set{b_j}$
  and $\gamma_{ij}= \alpha_i\otimes \beta_j$.
Then the matrix of $\phi$ in bases
$\set{\alpha_i}$ of $A^*$ and $\set{b_j}$ of $B$ is the matrix $(\gamma_{ij})$ of linear forms on $\PP(A\otimes B)$.
In a coordinate free way,
  over a point $[T]\in\PP(A\otimes B)$ for $\alpha\in A^*$,
  up to rescaling
  (which is possible to achieve consistently due to the twist by
    $\ccO_{\PP(A\otimes B)}(1)$),
  we have $\phi_{[T]}(\alpha) = \alpha \hook T \in B$,
  where $\hook$ denotes the tensor contraction
    $A^*\otimes (A\otimes B) \stackrel{\hook}{\longrightarrow} B$.

\begin{lemma}\label{lem_PAotimesBprime}
    Suppose $R$ is any scheme with a morphism  $f \colon R\to \PP(A\otimes B)$ and suppose $B' \subset B$ is a linear subspace.
    Then the following two conditions are equivalent:
    \begin{itemize}
      \item  the scheme-theoretic image of $f$ is contained in $\PP(A \otimes B')$,
      \item the pullback  $f^*\phi \colon \ccO_{R} \otimes A^* \to f^*\ccO_{\PP (A\otimes B)}(1) \otimes B$
            factorises through
            $\ccO_{R} \otimes A^* \to f^*\ccO_{\PP (A\otimes B)}(1) \otimes B'  \to   f^*\ccO_{\PP (A\otimes B)}(1) \otimes B$.
    \end{itemize}
\end{lemma}

\begin{proof}
    Consider $(B')^{\perp} = \set{\beta\mid \forall_{b'\in B'} \ \beta(b')=0} \subset B^*$.
    Then $A^*\otimes(B')^\perp \subset A^* \otimes B^*$ is the set of linear equations of $\PP(A\otimes B')\subset \PP (A\otimes B)$ and it generates the homogeneous ideal of this linear subspace.
    The scheme-theoretic image of $f$ is contained in
       $\PP(A \otimes B')$ if and only if $f^*(A^*\otimes(B')^\perp) = 0$.
    This vanishing occurs
       if and only if
       for all $\alpha\in A^*$ and $\beta\in (B')^\perp$
       one has $f^*(\alpha\otimes \beta) =0$.
     Equivalently, for all $\alpha\in A^*$ and $\beta\in (B')^\perp$
       we have $\beta\bigl((f^*\phi)(\alpha)\bigr)=0$,
       or more briefly $(B')^{\perp} (f^*\phi) =0$.
      The last vanishing happens if and only if the image of  $f^*\phi$ is contained in $f^*\ccO_{\PP (A\otimes B)}(1) \otimes B'$
       as claimed.
\end{proof}

\section{Cacti}
\label{sec_cacti}

Suppose $X\subset \PP(W)$ is a locally closed (hence quasiprojective) subscheme.
The main interest is in the case $X$ is a smooth projective variety. For technical reasons and arguments we also need the case when $X$ is not closed in $\PP(W)$, only locally closed.
Furthermore, some definitions and intermediate results are also valid in the more general setting of singular varieties and even quasiprojective schemes.

Recall, that for projective variety $X$ the \emph{Grassmann secant variety} is:
\[
  \sigma_{r,i}:=\overline{\set{E\in \Gr(i, W) \mid
    \exists \fromto{x_1}{x_r} \in X \text{ s.t.}  E \subset \linspan{\fromto{x_1}{x_r}} }}.
\]

\begin{defin}
\label{def_Grassmann_cactus_variety}
Fix integers $r$ and $i$.
The $r$-th \emph{Grassmann cactus variety} of $X$ is the closed subset $\cactus{r,i}{X} \subset \Gr(i,W)$ defined by:
\[
 \cactus{r,i}{X} := \overline{\set{E\in \Gr(i, W) \mid
    \exists \text{ subscheme } R\subset X \text{ s.t.} \dim R = 0, \deg R \leqslant r, E \subset \linspan{R} }}.
\]
The \emph{generic Grassmann cactus variety} $\ocactus{r,i}{X}$ of $X$ is:
\[
 \ocactus{r,i}{X} := \bigcap_{U\subset X} \cactus{r,i}{U}
\]
where the intersection is over all dense Zariski open subsets of $X$.
\end{defin}

In particular, if $i=1$ then $\cactus{r,i}{X} = \cactus{r}{X}$ is the cactus variety defined in \eqref{equ_define_cactus}.

\begin{thm}
   \label{thm_generic_cactus_and_smooth_locus}
   Suppose $X\subset \PP (W)$ is as above and $X_0\subset X$ is the
   smooth locus of $X$. Then:
   \[
     \ocactus{r,i}{X} = \cactus{r,i}{X_0}.
   \]
   In particular,
   if $X$ is smooth then for any open dense
   $U\subset X$ we have $\cactus{r,i}{U} =\cactus{r,i}{X}$.
\end{thm}
\begin{prf}
   Since $X_0 \subset X$ is an open subset, by definition we have 
   $\ocactus{r,i}{X} \subset \cactus{r,i}{X_0}$.
   Similarly to \eqref{equ_define_plain_cactus}, define:
   \[
    \cactusop{r,i}{X_0} := \set{E\in \Gr(i, W) \mid
    \exists \text{ subscheme } R\subset X_0 \text{ s.t.} \dim R = 0, \deg R \leqslant r, E \subset \linspan{R} }.
   \]
   Let $U \subset X$ be any open subset.
   Since $\cactus{r,i}{U}$ is closed and $\overline{\cactusop{r,i}{X_0}} = \cactus{r,i}{X_0}$, to prove the theorem it is enough to show that
   \[
     \cactusop{r,i}{X_0} \subset \cactus{r,i}{X_0\cap U}\subset \cactus{r,i}{U}.
   \]
   The second inclusion is clear. 
   To prove the first one,
   take any $E\in \cactusop{r,i}{X_0}$ so that there exists a finite subscheme $R\subset X_0$ of degree $r$ such that $E\subset \linspan R$.
   Consider a flat family of subschemes $\ccR\subset T \times X_0$ parametrised by a curve $T$ with a special point $0\in T$, such that special fibre is $R_0 = R$ and for general $t\in T$ the scheme $R_t$ is contained in $X_0\cap U$.
   This family exists
      by Proposition~\ref{prop_moving_scheme_into_generic_position}.
   Then we have:
   \[
      E \subset \linspan{R} = \linspan{\lim_{t\to 0} R_t}
        \subset \lim_{t\to 0} \linspan{R_t}.
   \]
   The last inclusion is the principle that
   \emph{the linear span of a limit is contained in the limit of the linear spans}.
   Although intuitively clear,
     a formal proof requires some work
   \cite[Prop.~5.29]{jabu_jelisiejew_finite_schemes_and_secants}.
   Note that the opposite inclusion may fail, as illustrated by famous example of three points in general position converging to three points on a line.

   Locally trivialising $\linspan{R_t}$ over general points $t$, we can find a subbundle $E_t\subset \linspan{R_t}$ of rank equal to $\dim E$, such that $E= \lim_{t\to 0} E_t$.
   Since each (general) $E_t$ is in $\cactus{r,i}{X_0\cap U}$, and since $\cactus{r,i}{X_0\cap U}$ is closed, it follows that also the limit $E$ is in the same cactus variety, concluding the proof.
\end{prf}

\section{Determinantal spines}

In this section we restrict to the cactus varieties
$\cactus{r}{X}= \cactus{r,1}{X}$
  as defined in \eqref{equ_define_cactus}, or equivalently,
  as in Definition~\ref{def_Grassmann_cactus_variety} for $i=1$.

\begin{thm}\label{thm_determinantal_cacti}
   Suppose $X\subset \PP(W)$ is a locally closed subscheme and $M$ is a matrix, whose entries are linear forms on $W$ (elements of $W^*$).
   If $X$ is contained in the locally closed locus
   \[
    \set{\rk M = k} = \set{\rk M \leqslant k} \setminus \set{\rk M \leqslant k-1} \subset \PP (W),
   \]
   then the cactus variety $\cactus{r}{X}$ is contained in the locus $ \set{\rk M \le k \cdot r}$.
\end{thm}

By \cite[Rem.~7.5]{galazka_mgr_publ}, if we only assume $X \subset \set{\rk M \leqslant k}$ the claim $\cactus{r}{X} \subset \set{\rk M \le k \cdot r}$ may fail.
Namely, \cite[Ex.~7.4]{galazka_mgr_publ} shows a singular variety $X$, a matrix such that $X\subset \set{\rk M \leqslant 1}$
and a point $F\in \cactus{2}{X}$ such that $\rk M(F)=3$.
Similar examples may be constructed if $X$ is nonreduced.
However, if $X$ is smooth,
then in Theorem~\ref{thm_determinantal_cacti_for_smooth} and Corollary~\ref{cor_determinantal_cacti_for_smooth_disconnected}
we show that such examples cannot be constructed.

\begin{prf}[ of Theorem~\ref{thm_determinantal_cacti}]
   We may assume $k>0$ as otherwise the statement is vacuous.
   The $a \times b$ matrix $M$ determines a linear map $g\colon W \to A\otimes B$,
     where $A$ is an $a$-dimensional vector space and $B$ is a $b$ dimensional vector space.
   We denote by the same letter the projectivised restriction $g\colon X \to \PP (A\otimes B)$.
   Note that $g\colon X \to \PP (A\otimes B)$ is a regular map, since $M(x)$ is never $0$ for $x\in X$.

   On $\PP(A\otimes B)$ we have the natural map of vector bundles
   \[
     \ccO_{\PP (A\otimes B)} \otimes A^* \to \ccO_{\PP (A\otimes B)}(1) \otimes B
   \]
   as in Section~\ref{sec_finite_schemes}.
   We denote by $\phi\colon \ccO_{X} \otimes A^* \to \ccO_{X}(1) \otimes B$ the pullback of this map to $X$.
   Note that our assumptions imply that $\phi$ has constant rank $k$.
   In particular, its image $\ccE$ is a vector bundle on $X$.

   Let $R \subset X$ be a finite subscheme of length $r$.
   Denote by $f\colon R \to \PP(A\otimes B)$ the composition of embedding $R \subset X$ and $g$.
   Consider the restriction $\ccE|_{R}$ as a subbundle of $ f^*\ccO_{\PP (A\otimes B)}(1) \otimes B$.
   By Lemma~\ref{lem_exists_B_prime}, there exists a linear subspace of dimension at most $k \cdot r$
      such that $\ccE|_{R} \subset f^*\ccO_{\PP (A\otimes B)}(1) \otimes B'$.
   By Lemma~\ref{lem_PAotimesBprime} this implies that $f (R)\subset \PP(A\otimes B')$.

   Thus $\langle R \rangle$ is contained in the linear space $g^{-1} (\PP(A\otimes B')) \subset \PP(W)$.
   Next, since $\dim B'\le k\cdot r$, the linear space $g^{-1} (\PP(A\otimes B'))$ is contained in the locus  $\set{\rk M\le k \cdot r}$.
   Therefore
   \[
    \cactusop{r}{X} = \bigcup \set{\langle R \rangle \mid
     R\subset X, \dim R = 0, \deg R \leqslant r}
     \subset \set{\rk M\le k \cdot r},
   \]
      and since $\set{\rk M\le k \cdot r}$
      is closed the claim of the theorem follows:
   \[
      \cactus{r}{X} = \overline{\cactusop{r}{X}} \subset
      \overline{\set{\rk M\le k \cdot r}} =\set{\rk M\le k \cdot r}.
   \]
\end{prf}

Next we focus on the case of $X$ being smooth.
The main examples include Segre-Veronese varieties or Grassmannians.

\begin{thm}
   \label{thm_determinantal_cacti_for_smooth}
   Suppose $X\subset \PP(W)$ is a smooth connected quasiprojective variety and $M$ is a matrix of linear forms.
   If $X\subset \set{\rk M \leqslant k}$,
   then $\cactus{r}{X} \subset \set{\rk M \le k \cdot r}$.
\end{thm}

\begin{prf}
  Argue by induction on $k$.
  If $k=0$, that is if $M|_X \equiv 0$,
    then for any point $p\in\linspan{X}$
    the evaluated matrix $M(p)$ is also $0$,
    because the entries of $M$ are linear.
  Clearly, $\cactus{r}{X}$ is contained in the linear span of $X$, hence the conclusion.

  For any $k>0$, let $U\subset X$ be the open subset
  $U:=\set{\rk M = k} \cap X$.
  If $U$ is empty,
    then $X\subset \set{\rk M \leqslant k-1}$,
    and by induction,
  $\cactus{r}{X}\subset \set{\rk M \le (k-1) \cdot r} \subset \set{\rk M \le k \cdot r}$.
  Thus we can assume $U\neq \emptyset$
  and therefore $U$ is dense in $X$.
  By Theorem~\ref{thm_generic_cactus_and_smooth_locus} and the definition of $\ocactus{r}{X}$ we have:
  \[
     \cactus{r}{X} = \ocactus{r}{X} = \ocactus{r}{U}
     = \cactus{r}{U}.
  \]
  By Theorem~\ref{thm_determinantal_cacti} we also have
  $\cactus{r}{U} \subset \set{\rk M \le k \cdot r}$.
  Combining the above, we get the claim $\cactus{r}{X} \subset \set{\rk M \le k \cdot r}$.
\end{prf}

In the final steps we generalise Theorem~\ref{thm_determinantal_cacti_for_smooth} to the case where $X$ is not connected.
First recall, that for two reduced subschemes $Y, Z\subset \PP(W)$,
their join $Y+Z\subset \PP(W)$ is defined as:
\[
 Y+Z := \overline{\bigcup\set{\linspan{y, z} \colon y\in Y, z\in Z}}.
\]
We choose the convention that $Y+\emptyset=\emptyset +Y =Y$
  (we treat $\emptyset\subset \PP (W)$
   as the projectivisation of $\set{0} \subset W$).
Moreover, consistently with the convention in Definition~\ref{def_ranks} we set $\cactus{0}{X}=\emptyset$.

\begin{lemma}
     \label{lem_cactus_as_join}
   Suppose $X_i\subset \PP(W)$ for $i=1, 2$
      are two locally closed subschemes
      and that $X_1\cap X_2 =\emptyset$.
   Let $X= X_1 \cup X_2 \subset \PP(W)$ and assume that $X_i \subset X$ is open for both $i=1, 2$.
   Then:
   \[
     \cactus{r}{X} = \bigcup_{q=0}^r  \Bigl(\cactus{q}{X_1}+\cactus{r-q}{X_2}\Bigr).
   \]
\end{lemma}

\begin{prf}
   The $\supset$ inclusion is clear.
   Also both sides of the equation are closed, hence it is enough to prove
   \[
    \cactusop{r}{X} =\bigcup \set{\langle R \rangle \mid
     R\subset X, \dim R = 0, \deg R \leqslant r} \subset \bigcup_{q=0}^r  \bigl(\cactus{q}{X_1}+\cactus{r-q}{X_2}\bigr).
   \]
   Thus it suffices to show that for each subscheme $R\subset X$
   we have $\langle R \rangle \subset \cactus{q}{X_1}+\cactus{r-q}{X_2}$ for some $q$.
   Indeed, let $R_i = R\cap X_i$ and let $q = \deg R_1$.
   Since $X_1$ and $X_2$ are disjoint, and $X_i\subset X$ are open we must have $\deg R_2= r-q$.
   Therefore $\linspan{R_1}\subset \cactus{q}{X_1}$
     and $\linspan{R_2}\subset \cactus{r-q}{X_2}$.
   Further:
   \[
      \linspan{R} = \linspan{R_1}+ \linspan{R_2} \subset \cactus{q}{X_1} + \cactus{r-q}{X_2}
   \]
   as claimed.
\end{prf}

\begin{cor}
   \label{cor_determinantal_cacti_for_smooth_disconnected}
   Suppose $X\subset \PP(W)$ is a smooth locally closed subscheme  and $M$ is a matrix of linear forms.
   If $X\subset \set{\rk M \leqslant k}$,
   then $\cactus{r}{X} \subset \set{\rk M \le k \cdot r}$.
\end{cor}
\begin{prf}
   We argue by induction on the number $n$ of connected components of $X$.
   If $X$ is connected, that is $n=1$, then the claim is identical to Theorem~\ref{thm_determinantal_cacti_for_smooth}.
   If $n\geqslant 2$ then write $X= X_1 \sqcup X_2$ with both $X_i$ nonempty and open in $X$.
   Then by the inductive assumption
   $\cactus{q}{X_1} \subset \set{\rk M \le k \cdot q}$ and
   $\cactus{r-q}{X_2} \subset \set{\rk M \le k \cdot (r-q)}$ for any $q \in \setfromto{0}{r}$.
   For any points
   $F_1\in \cactus{q}{X_1}$ and $F_2 \in \cactus{r-q}{X_2}$,
   we have $\linspan{F_1, F_2} \subset \set{\rk M \le k \cdot r}$
   and hence $\cactus{q}{X_1}+\cactus{r-q}{X_2} \subset \set{\rk M \le k \cdot r}$.
   Therefore, the claim of the Corollary follows from Lemma~\ref{lem_cactus_as_join}.
\end{prf}

We rephrase Corollary~\ref{cor_determinantal_cacti_for_smooth_disconnected} as a barrier-type statement.

\begin{cor}
   Suppose $X\subset \PP(W)$ is smooth locally closed and nondegenerate (that is, not contained in any hyperplane),
   and let $g$ be an integer such that $\cactus{g}{X}=\PP (W)$.
   Assume $M$ be a matrix of linear forms.
   Let $k:=\max\set{\rk M(x) \mid x\in X}$.
   Then for every $p\in \PP(W)$ we have $\rk M(p)  \leqslant k \cdot g$.
\end{cor}
In particular, linear rank methods for bounding border rank cannot yield results beyond $r$ in which $\cactus{r}{X}$ fills in the ambient space.

Theorem~\ref{thm_cactus_barriers_intro} follows from Corollary~\ref{cor_determinantal_cacti_for_smooth_disconnected}, as the smooth locus $X_0$ will be a disconnected union of irreducible open subsets $X$.

\appendix
\section{Appendix: from sets to algebraic varieties}
\label{sec_appendix}

Suppose $\hat X \subset W$ is any subset of a vector space $W$ and for any $F\in W$ define:
\[
  r_{\hat X}(F):= \min \set{r \mid F \in \linspan{\fromto{x_1}{x_r}} \text{ for } x_i \in \hat{X}}
\]
analogously to Definition~\ref{def_ranks}.
In this appendix we briefly explain that the barriers to linear rank methods also apply to this setting, where $\hat X$ is not necessarily an algebraic set or variety.

Indeed, let $\kk^*\cdot \hat X\subset W$ be the set obtained as rescaling every element in $\hat{X}$.
Note that $r_{\hat X}(F)= r_{\kk^* \cdot \hat X}(F)$ for any $F\in W$.
Further let $X\subset \PP (W)$ be the set of projecive classes of elements in $\hat{X}$ (or equivalently, in $\kk^*\cdot \hat{X}$).
Finally, define $\overline{X}$ to be the Zariski closure of $X$ in $\PP(W)$.

Then for any matrix $M$ of linear forms on $W$ and any $k$ we have:
\[
   \hat{X}\subset \set{\rk M \leqslant k}
   \Leftrightarrow
   \kk^*\cdot \hat{X} \subset \set{\rk M \leqslant k}
   \Leftrightarrow
   X \subset \set{\rk M \leqslant k}
   \Leftrightarrow
   \overline{X} \subset \set{\rk M \leqslant k}.
\]
The first equivalence holds by linearity of entries of $M$,
the second is just descending to the projective space, and finally the third equivalence follows since $\set{\rk M \leqslant k}$ is Zariski closed.
But now $\overline{X}$ is algebraic set (reduced scheme), and we can apply the barrier theorems, for instance Theorem~\ref{thm_cactus_barriers_intro}. We conclude that using linear rank methods we will never get a better lower bound for $\hat{X}$-rank than generic cactus rank of $(\overline{X})_0$, the smooth locus of $\overline{X}$.

\bibliography{cactus_barriers.bbl}

\end{document}